\documentclass[a4paper,12pt]{article}
    \usepackage[top=2cm,bottom=2cm,left=2.5cm,right=2.5cm]{geometry}
    \usepackage{cite, amsmath, amssymb}
    \usepackage[margin=0.5cm,%
                font=small,%
                format=hang,%
                labelsep=period,%
                labelfont=bf]{caption}
\usepackage{amsmath}
\usepackage{amsfonts}
\usepackage{amssymb}
\usepackage{cite}
\usepackage{amsthm}
\usepackage{makeidx}
\usepackage{graphicx}
\usepackage{mathrsfs,xcolor,multicol}
\usepackage{enumerate}
\usepackage{blkarray}
\usepackage{bm}
\usepackage{setspace}
\usepackage{float}
\usepackage{authblk}
\usepackage{multirow}
\usepackage{booktabs}
\usepackage[ruled,vlined]{algorithm2e}
\usepackage{blkarray}
\usepackage{authblk}
\usepackage{multirow}
\usepackage{lineno}

\usepackage{blkarray, bigstrut}
\usepackage{authblk}
\usepackage{multirow}

\usepackage{fancyhdr}
\usepackage{lastpage}

\usepackage{amsmath, amssymb, booktabs, lastpage,}
\usepackage{multicol}
\usepackage{hyperref}
\usepackage{blkarray}
\usepackage{empheq} 

\usepackage{tikz-cd}
\usepackage{tikz}
\usetikzlibrary{arrows}
\usetikzlibrary{automata,positioning}
\usetikzlibrary{circuits.logic.US,circuits.logic.IEC,fit}
\newcommand*\circled[1]{\tikz[baseline=(char.base)]{
            \node[shape=circle,draw,inner sep=2pt] (char) {#1};}}

\numberwithin{table}{section}
\numberwithin{equation}{section}

\theoremstyle{plain}
\newtheorem{theorem}{Theorem}[section]
\newtheorem{proposition}[theorem]{Proposition}
\newtheorem{definition}[theorem]{Definition}

\newtheorem{example}[theorem]{Example}

\newtheorem{remark}[theorem]{Remark}

\usepackage{authblk}
\author[1]{ \textbf{Bryan S. Hernandez}}
\author[2]{ \textbf{Patrick Vincent N. Lubenia}}

\affil[1]{\small \textit{Institute of Mathematics, University of the Philippines Diliman, Quezon City 1101, Philippines}}
\affil[2]{\small \textit{Systems and Computational Biology Research Unit, Center for Natural Sciences and Environmental Research, Manila 0922, Philippines}}
\affil[*]{Email addresses: \texttt{bshernandez@up.edu.ph},
\texttt{pnlubenia@upd.edu.ph}}

\title{\textbf{A decomposition-based approach for deriving positive steady states of a class of chemical reaction networks with non-mass-action kinetics}}
\date{}

\begin{document}
\maketitle
\begin{abstract} 
Steady states are frequently used to investigate the long-term behaviors of (bio)-chemical systems. Recently, there has been a growing interest in network-based approaches due to their efficiency in deriving parametrizations of positive steady states in systems with mass-action kinetics. In this study, we extend this approach to derive positive steady states in networks under non-mass-action kinetics, specifically mixed kinetics.
In a system with mixed kinetics, some reactions {may follow} mass-action kinetics, while others in the same network follow different rate laws, such as quotient rate laws. An example of such complexity is evident in a mathematical model of the insulin signaling pathway in type 2 diabetes. To compute its positive {steady states}, we adapt our existing network decomposition approach, originally designed for mass-action kinetics, to handle networks with non-mass-action kinetics. This approach involves breaking down a given network into smaller, independent subnetworks to derive the positive steady states of each subnetwork separately. These individual steady states are then combined to obtain the positive steady states of the entire network.
This strategy makes computations more manageable for complex and large networks.
More importantly, this method could separate reactions with purely mass-action kinetics into certain subnetworks from those that follow different rate laws.
We also present an illustrative example that provides insights into methods for transforming networks with mixed kinetics into their associated mass-action systems.
	
\end{abstract}

\thispagestyle{empty}
\section{Introduction}
\label{sec:1}

In the past decades, significant attention has been given to \emph{chemical reaction networks (CRNs)} to study dynamical behaviors of (bio)chemical systems. In particular, \emph{steady states} are often used to describe the long-term behaviors of these systems.
{To parametrize the positive steady states of mass-action systems, in particular, the method of network translation \cite{HongSIAM,Johnston2014} that modifies the structure of the network, can be employed \cite{MJEB2019,JMP2019:parametrization}.}By ``parametrization'' we simply mean expressing the steady states of a network in terms of the parameters of the model.
Furthermore, a network obtained after applying the network translation is called a \emph{translated network}.

To illustrate the process of network translation, consider a very simple CRN with only two reactions $0 \to A$ (production of $A$) and $A+B \to B$ (interaction between $A$ and $B$ results in the disappearance of $A$). The reaction $A+B \to B$ can be shifted to $A \to 0$ (i.e., replacing an original reaction with the same stoichiometric vector) while still associating it with the rate law of the original second reaction based on the complex $A+B$ (i.e., $k_2 ab$) to maintain the dynamics of the system.
{We clarify that $B$ does not disappear in the network that is why the concentration $b$ still appears in the rate function. Such shifting is done to obtain a network with desirable properties needed for the method of network translation to work but still preserving the stoichiometric vectors in the network, i.e., both $A+B \to B$ and $0 \to A$ just loses a single $A$.}
Hence, network translation produces two structures. One structure, called {the} \emph{stoichiometric network}, is a network with nodes that are identified by the new reactions. The other structure, called \emph{kinetic-order network}, is a network with nodes coming from the old or original source nodes which impose the original kinetics. A network with these two structures is called a \emph{generalized} (\emph{chemical reaction}) \emph{network} \cite{Muller2012,Muller2014}. {Further details are given in Section \ref{section:network:translation}.}

Specifically, if one is able to determine a translated network of a CRN with mass-action kinetics where the underlying network has weak reversibility and zero deficiency (i.e., both stoichiometric and kinetic-order networks are weakly reversible and zero deficiency), then one can easily derive the parametrization of the network's positive state using the method introduced by Johnston et al. \cite{JMP2019:parametrization}. 
\emph{Weak reversibility} means that each reaction is {on a directed} cycle when the CRN is regarded as a directed graph, while \emph{deficiency} measures the linear dependency among the reactions in the network \cite{ShinarFeinberg2011}.
{Existence of positive steady states for weakly reversible mass-action systems were studied in \cite{bboros2019}.}

Recently, Hernandez et al. \cite{Hernandez2023} proposed a significantly {more} efficient way of solving positive steady states for networks with mass-action kinetics that can be decomposed into independent subnetworks. The method first decomposes the network into independent subnetworks and then parametri- zes the positive steady states of the subnetworks individually using the method of Johnston et al. \cite{JMP2019:parametrization} (rather than directly parametrizing the positive {steady states} of the whole network). Finally, the positive steady states of the subnetworks are combined to derive the parametrized positive steady states of the whole network. 

In this work, we modify and extend the method of Hernandez et al. \cite{Hernandez2023} to accommodate networks that follow {non-mass-action} kinetics (e.g., polynomial, Michaelis-Menten, Hill-type, quotient, or mixed kinetics). For subnetworks that endowed with purely mass-action or power-law kinetics, we still apply the method proposed by Johnston et al. \cite{JMP2019:parametrization}. On the other hand, for subnetworks that follow a different kinetics, we compute manually 
for the {steady states}. We apply our method to derive the parametrized positive {steady states} of a complex mathematical model of an insulin signaling pathway in type 2 diabetes {\cite{BrannmarkInsulin2013}}.
{Parametrizations of positive steady states are useful in assessing essential biological characteristics like absolute concentration robustness and multistationarity, as documented or mentioned in various studies \cite{craciun2006multiple, dickenstein2019multistationarity,ShinarFeinberg2010,JMP2019:parametrization}.
}

\section{Preliminaries}
\label{chap:preliminaries}

    In this section, we discuss the basic and important concepts on CRNs and chemical reaction systems \cite{FeinbergLecture,FeinbergBook2019}. We also present useful concepts and results regarding decomposition of CRNs \cite{Feinberg1987stability,FeinbergBook2019,Hernandezetal2022,Hernandez2023}.
    
    \subsection{Notations}
    Let $\mathbb R_{\ge 0}$ denote the \emph{set of non-negative real numbers}, and $\mathbb R_{> 0}$ the \emph{set of positive real numbers}. Similarly, let $\mathbb Z_{\ge 0}$ be the \emph{set of non-negative integers}.
    
	\subsection{Chemical reaction networks}
	\begin{definition}
		A {\emph{chemical reaction network}} $\mathcal{N}$ is a triple of nonempty and finite sets $\left(\mathcal{S},\mathcal{C},\mathcal{R}\right)$ where
		\begin{itemize}
		\item[1.] $\mathcal{S}$ is the set of \emph{species} $X_1,\ldots,X_m$,
		\item[2.] $\mathcal{C}$ is the set of \emph{complexes} of the form $y=\displaystyle \sum_{i=1}^m y_i X_i$ with $y_i\in \mathbb{Z}_{\ge0}$, and
		\item[3.] $\mathcal{R} \subset \mathcal{C} \times \mathcal{C}$ is the set of \emph{reactions} that satisfies the following properties:
				\begin{itemize}
		        \item[a.] $(y,y) \notin \mathcal{R}$ for each $y \in \mathcal{C}$, and
		        \item[b.] for each $y \in \mathcal{C}$, there exists {a} $y' \in \mathcal{C}$ such that $(y,y') \in \mathcal{R}$ or $(y',y) \in \mathcal{R}$.
		        \end{itemize}
		\end{itemize}
	\end{definition}

We often use the notation $y\to y'$ to denote the reaction $(y,y')$. In the definition, $m$ is the number of species in the network, while we let $n$ and $r$ be the numbers of complexes and reactions, respectively. {In the reaction $y\to y'$, the complex $y$ is called a \emph{reactant} or \emph{source complex} while the complex $y'$ is called a \emph{product complex}. Furthermore, we can identify the complexes with
vectors in $\mathbb{R}^m$.}

One can view a CRN as a \emph{directed graph} with complexes as \emph{vertices} and reactions as \emph{edges}. The {(\emph{strong}) \emph{linkage classes}} of the CRN are the (strongly) connected components of the graph.
A CRN is {\emph{weakly reversible}} if each of its linkage classes is a strong linkage class.
Equivalently, it is weakly reversible if each reaction belongs to a directed cycle.

\begin{example}
Consider the CRN in Figure \ref{fig:CRNexample}, which we denote by $\mathcal{N}$. This network has three species ($A, B,$ and $C$), seven complexes ($B+C$, $A+C$, $A$, $0$, $B$, $2C$, and $C$) and five reactions ($R_1$, $R_2$, $R_3$, $R_4$, and $R_5$). In particular, for reaction $R_1: B+C \to A+C$, $B+C$ is its source complex while $A+C$ is its product complex. {In addition, $R_2: A \to 0$ is an outflow reaction that denotes degradation or consumption of species $A$. Specifically, the zero complex (0) can be interpreted as the exterior of the network environment. On the other hand, $R_3: 0 \to B$ is an inflow of reaction that denotes production or supply of species $B$.}

The CRN has three linkage classes because there are three connected components. Furthermore, $\mathcal{N}$ is not weakly reversible because there are reactions that do not belong to a cycle (e.g., $R_1: B+C \to A+C$).
\label{ex:CRN}
\end{example}

\begin{figure}[!h]
{
\footnotesize
\begin{center}
\includegraphics[width=11cm,height=8cm,keepaspectratio]{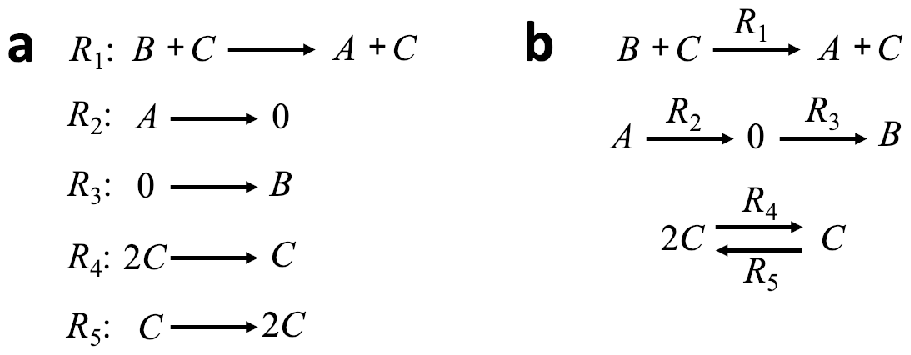}
\end{center}
}
\caption{{\bf{A simple example of a CRN.}}
{
{\bf{a}} The CRN has three species ($A, B,$ and $C$), seven complexes ($B+C$, $A+C$, $A$, $0$, $B$, $2C$, and $C$) and five reactions ($R_1$, $R_2$, $R_3$, $R_4$, and $R_5$).
{\bf{b}} The representation of the CRN where the reactions are grouped according to their respective linkage classes.
}
}
\label{fig:CRNexample}
\end{figure}

From a dynamical perspective, the concentrations of the species can vary when a reaction $y \to y'$ occurs at a certain time. The change can be quantified by the difference $y'-y$, which is called the \emph{reaction vector} of the reaction $y \to y'$. All the changes caused by the reactions belong to the subspace of the ambient space $\mathbb{R}^m$, known as the \emph{stoichiometric subspace} of a CRN defined as $$S := {\rm span}\left\{{y' - y| y \to y' \in \mathcal{R}}\right\}.$$
The $m \times r$ matrix where the $i$-th column contains the coefficients of the associated species in the $i$-th reaction vector is called the \emph{stoichiometric matrix}.

We now introduce an important concept
in the theory of CRNs:
	\begin{definition}
	The \emph{deficiency} of a CRN is $\delta=n-\ell-s$ where $n$ is the number of complexes, $\ell$ is the number of linkage classes, and $s$ is the dimension of the stoichiometric subspace (or the rank of the stoichiometric matrix).
	\end{definition}

\begin{example}
Consider the CRN in Example \ref{ex:CRN}. The reaction vector associated to a reaction in the CRN can be obtained by subtracting from the product complex the reactant complex. For instance, {the} reaction vector associated with the first reaction $B+C \to A+C$ can be obtained by subtracting $B+C$ (source complex) from $A+C$ (product complex). Hence, the first reaction vector is $(A+C)-(B+C)=A-B$. By associating species $A$, $B$, and $C$ with the standard basis vectors of the Euclidean space $\mathbb{R}^3$ (three corresponds to the number of species), $A-B=1A-1B+0C=1\cdot[1,0,0]^\top-1\cdot[0,1,0]^\top+0\cdot[0,0,1]^\top=[1,-1,0]^\top$. This is precisely the first column of the stoichiometric matrix of the CRN. The same procedure is done to obtain all the reaction vectors, and hence, all the columns of the stoichiometric matrix. The rank of this stoichiometric matrix is three ($s=3$) (one can see this be inspection in Figure \ref{fig:stoichiometric} (right)).
From Figure \ref{fig:CRNexample}b, there are seven complexes ($n=7$) and three linkage classes ($\ell=3$). Thus, the deficiency is $\delta = n -\ell -s = 7-3-3=1$.
\label{ex:stochiometric}
\end{example}

\begin{figure}[!h]
{
\footnotesize
\begin{center}
\includegraphics[width=14cm,height=7cm,keepaspectratio]{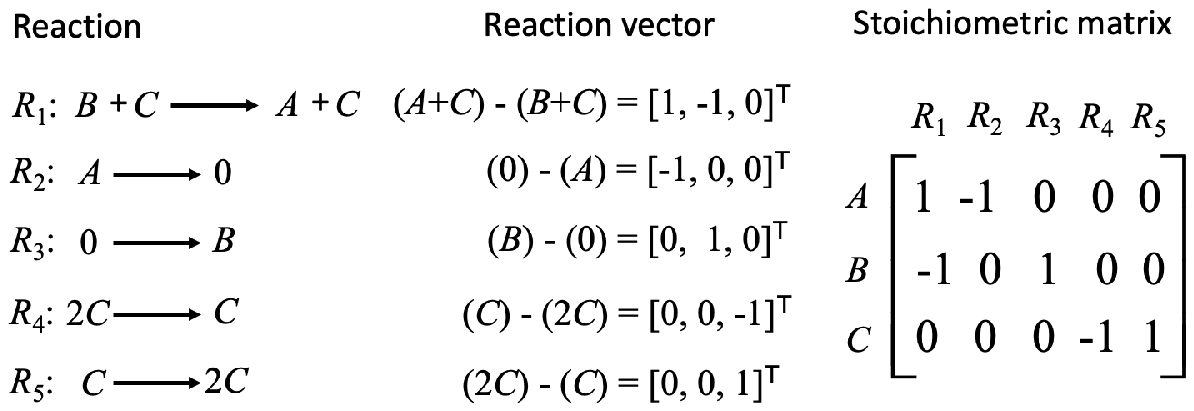}
\end{center}
}
\caption{{\bf{The stoichiometric matrix of the simple CRN in Figure \ref{fig:CRNexample}.}}
{A reaction vector is obtained by subtracting the source complex from the product complex of a reaction. After getting all the reaction vectors, the stoichiometric matrix is constructed by placing these vectors as columns, following the order of the reactions in the network. 
}
}
\label{fig:stoichiometric}
\end{figure}

We associate kinetics with a CRN to describe the dynamics of a given system. We define the kinetics and, consequently, a chemical reaction system in the following manner.

\begin{definition}
A \emph{kinetics} for a reaction network $\mathcal{N}=(\mathcal{S}, \mathcal{C}, \mathcal{R})$ is an assignment to
each reaction $y \to y' \in \mathcal{R}$ of a continuously differentiable {rate function} $\mathcal{K}_{y\to y'}: \mathbb{R}^\mathcal{S}_{{\geq} 0} \to \mathbb{R}_{\ge 0}$ such that the following positivity condition holds:
$\mathcal{K}_{y\to y'}(c) > 0$ if and only if ${\rm{supp \ }} y \subset {\rm{supp \ }} c$, where ${\rm{supp \ }} y$ refers to the support of the vector $y$.
The system $\left(\mathcal{N},\mathcal{K}\right)$ is called a \emph{chemical kinetic system} or a \emph{chemical reaction system}.
\end{definition}

\begin{definition}
	The \emph{species formation rate function} of a chemical reaction system $(\mathcal{N},\mathcal{K})$ is defined as $f\left( x \right) = \displaystyle \sum\limits_{{y} \to {y'} \in \mathcal{R}} {{\mathcal{K}_{{y} \to {y'}}}\left( x \right)\left( {{y'} - {y}} \right)}.$
\end{definition}
The system of \emph{ordinary differential equations} (ODEs) of a chemical reaction system is given by $\dfrac{{dx}}{{dt}} = f\left( x \right)$. In addition, a \emph{steady state} or an \emph{equilibrium} of the system is a vector of concentration of species that makes $f$ the zero vector. The set of positive steady states of the network $\mathcal{N}$ with specified kinetics $\mathcal{K}$ is denoted by $E_+:=E_+(\mathcal{N},\mathcal{K}).$

\begin{definition}
A kinetics for a CRN $(\mathcal{S},\mathcal{C},\mathcal{R})$ is \emph{mass-action} if for each reaction $y\to y'$ (i.e., $[y_1,y_2,\ldots, y_m]^\top \to [y_1',y_2',\ldots, y_m']^\top$),
$$\mathcal{K}_{y\to y'}(x)=k_{y\to y'}\prod _{i \in \mathcal{S}} x_i^{y_i}$$
for some $k_{y\to y'}>0$.
\end{definition}

For example, suppose that a CRN that has mass-action kinetics contains the reaction $A+B \to B$. The associated kinetics (or rate law) for the reaction is $kab$, i.e., the product of the concentrations $a$ and $b$ of species $A$ and $B$, respectively (determined by the source complex $A+B$) multiplied by the rate constant $k$.

A \emph{mass-action system} is a CRN endowed with mass-action kinetics. If at least one reaction does not follow the mass-action rate function, then the system is called \emph{non-mass-action}. For instance, if the rate function of a certain reaction in a network follows a quotient rate law, then the corresponding system is not a mass-action system.

\begin{example}
Consider again the CRN presented in Examples \ref{ex:CRN} and \ref{ex:stochiometric}. Figure \ref{fig:ODEs} shows how to obtain the ODEs of the simple CRN under mass-action kinetics. The concentrations of the species $A$, $B$, and $C$ are denoted as $a$, $b$, and $c$, respectively. To get the rate function for reaction $R_1:B+C \to A+C$, we multiply the concentrations of the species present in the source complex, i.e., $bc$. We then multiply it by the rate constant $k_1$ that we associate with the reaction $R_1$. Thus, the rate function for the first reaction is $k_1bc$. A similar procedure can be applied to obtain the rate functions of the remaining reactions.

To obtain the ODEs associated with a CRN with mass-action kinetics, we multiply the reaction rate by the reaction vector, e.g., $k_1bc \cdot [1,-1,0]^\top$. We do this for each reaction and get the sum over all these reactions. We then equate the result to the vector of time derivatives of the concentrations of the species (Figure \ref{fig:ODEs} bottom). Simplifying, we get the following system of ODEs:
\begin{align*}
    \dfrac{da}{dt}=& k_1 bc -k_2 a\\
    \dfrac{db}{dt}=& -k_1 bc +k_3 \\
    \dfrac{dc}{dt}=& -k_4 c^2 +k_5 c.
\end{align*}
\end{example}

\begin{figure}[!h]
{
\footnotesize
\begin{center}
\includegraphics[width=20cm,height=10cm,keepaspectratio]{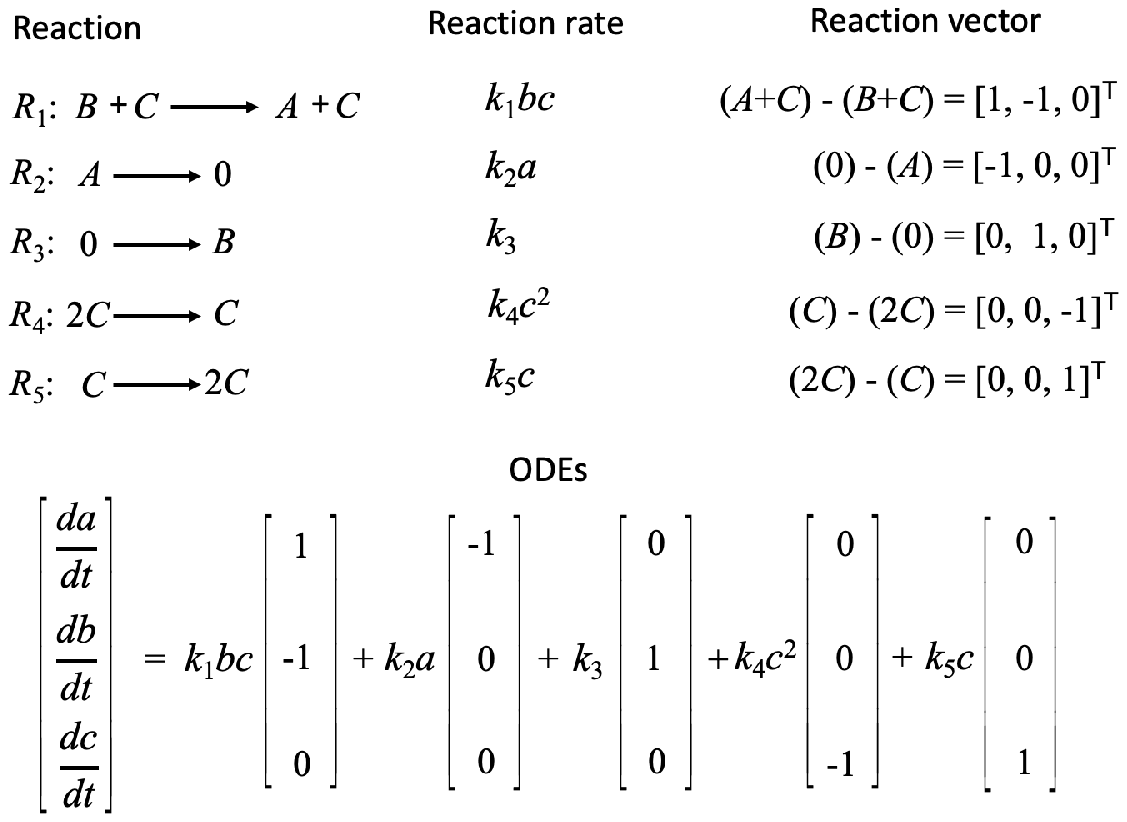}
\end{center}
}
\caption{{\bf {The ODEs of a CRN (in Figure \ref{fig:CRNexample}) with mass-action kinetics.}}
Assume that the CRN follows the mass-action kinetics. To obtain the ODEs, we need to compute the reaction rate and the reaction vector for each reaction. The left-hand side of the ODEs is composed of the time derivatives of the concentration of the species. On the other hand, to get the right-hand side of the ODEs, multiply the reaction vector by the associated reaction rate of each reaction. Then, get the sum of all these products over all the reactions.
}
\label{fig:ODEs}
\end{figure}

\subsection{Decomposition of chemical reaction networks}
\label{CRN:decomposition}

We start this section with a formal definition of the general decomposition of reaction networks (Appendix 6.A \cite{FeinbergBook2019}, Section 5.4 \cite{Feinberg1987stability}).

\begin{definition}
    Let $\mathcal{N}$ be a reaction network and $\mathcal{R}$ its reaction set.
	A {\emph{decomposition of the reaction network}} into $\mathcal{N}_1$, $\mathcal{N}_2$, \ldots, $\mathcal{N}_\alpha$ is induced by a partition of $\mathcal{R}$ into $\mathcal{R}_1,  \mathcal{R}_2, \ldots, \mathcal{R}_\alpha$, respectively. The resulting networks $\mathcal{N}_1$,  $\mathcal{N}_2$, \ldots, $\mathcal{N}_\alpha$ are called \emph{subnetworks} of $\mathcal{N}$.
\end{definition}


Technically, we consider each $\mathcal{N}_i$ to have the same set of species as $\mathcal{N}$, although some of the species seem to have no role or {be} absent in this particular subnetwork. In addition, we consider $\mathcal{N}_i = (\mathcal{S},\mathcal{C}_i,\mathcal{R}_i)$ where $\mathcal{C}_i$ is the set of all complexes that appear as a product or reactant of a reaction in the subnetwork $i$ \cite{FeinbergBook2019}. If $\mathcal{N}$ has only one subnetwork under the decomposition, we say that it has a {\emph{trivial decomposition}}.
In this case, $\alpha=1$ and the network itself is its only subnetwork.

We note that each subnetwork $\mathcal{N}_i$ has its own stoichiometric subspace $S_i\subseteq S$. We are now ready to define the concept of independent decomposition:

\begin{definition}
A network decomposition is said to be {\emph{independent}} if the stoichiometric subspace of the whole network is equal to the direct sum of the stoichiometric subspaces of the subnetworks.
\end{definition}

An equivalent way of showing independence of a CRN decomposition is to show that the rank of the stoichiometric matrix is equal to the sum of the ranks of the stoichiometric matrices of the subnetworks.

We are interested with independent decompositions of CRNs because the set of positive steady states of the whole network is equal to the intersection of the sets of positive states of the independent subnetworks as given in Theorem \ref{feinberg:decom:thm} in Appendix \ref{Feinberg:Decomposition:Theorem}.

Hernandez et al. \cite{Hernandezetal2022,HDLC2021} developed a method of determining the finest independent decomposition, i.e., independent decomposition with the highest number of subnetworks. {It was shown that this finest independent decomposition is unique \cite{Hernandezetal2022}. This method is simple and efficient as we partition the reaction set via partitioning the rows of the associated stoichiometric matrix of the network.}
Additionally, programs developed using Octave and MATLAB were created by P. Lubenia so that one can easily obtain the finest independent decomposition by providing the reactions of the given network \cite{LubeniaINDECS}.

\subsection{Translated networks}
\label{section:network:translation}
We now discuss the notion of the \emph{generalized chemical reaction networks} (GCRNs) developed by Stefan M\"uller and Georg Regensburger \cite{Muller2012,Muller2014}.
Suppose that $G$ is a directed graph with vertex set $V$ and edge set $E \subseteq V \times V$. On an edge $i\to j:=(i,j)$, $i$ is called the \emph{source vertex}. Denote by $V_s$ the set of all source vertices in $V$.
We define a \emph{GCRN} as a directed graph $G$, with vertex set $V$, together with two maps: (i) $y: V \to \mathbb{R}^m_{\ge 0}$ that assigns to each vertex a \emph{stoichiometric complex}, and (ii) $\widetilde{y}: V_s \to \mathbb{R}^m_{\ge 0}$ that assigns to each vertex a \emph{kinetic complex}.

We also introduce the concept of network translation by Matthew Johnston \cite{Johnston2014}.
A GCRN $\mathcal{N}'$ is said to be a \emph{network translation} of a CRN $\mathcal{N}$ if the reaction vectors of the
original reactions are preserved,
and the original source complexes of $\mathcal{N}$ are transferred as kinetic complexes in $\mathcal{N}'$ \cite{Johnston2014,JMP2019:parametrization}.
The generalized network $\mathcal{N}'$ is also called a \emph{translated network}.
Therefore, the original CRN $\mathcal{N}$ and the translated network $\mathcal{N}'$ have the same set of ODEs, i.e., they are dynamically equivalent.

For example, suppose that a network $\mathcal{N}_E$ has only one reaction $A \to 2A$. We form $\mathcal{N}_E'$ by replacing the original reaction by $0 \to A$ with the same reaction vector as the original, i.e., $A-0=2A-A$. In addition, we consider the original source complex ($A$) of the original reaction $A \to 2A$ to dictate the rate function ($ka$) to maintain the dynamics. Thus, the translated network $\mathcal{N}_E$ is dynamically equivalent to $\mathcal{N}_E'$.

As a result, a translated network is a GCRN with two associated structures: the \emph{stoichiometric CRN} and the \emph{kinetic-order CRN}. The deficiencies of the stoichiometric and kinetic-order CRNs are called \emph{effective deficiency} and \emph{kinetic deficiency}, respectively.

When both the effective and kinetic deficiencies of a weakly reversible translated network of a CRN are both zero, then we can parametrize the positive steady states of a CRN efficiently with a formula given in
Theorem \ref{thm:parametrization} in Appendix \ref{GCRN}.

\subsection{Previous works and relation to the current work}

Recently, Hernandez et al. \cite{Hernandez2023} proposed a framework to parametrize the positive steady states of CRNs, endowed with mass-action kinetics, efficiently via network decomposition. Their method was illustrated to be useful for power-law systems \cite{Hernandez2023powerlaw}, a larger class of kinetic systems which contains the mass-action kinetics. Furthermore, Villareal et al. \cite{Villareal2024} extended the method of parametrization for CRNs that can be decomposed into $n$ identical and independent subnetworks for any positive integer $n$.

Based on Theorem \ref{feinberg:decom:thm} in Appendix \ref{Feinberg:Decomposition:Theorem}, 
the set of positive steady states of a given network is the intersection of the sets of positive steady states of its subnetworks when the underlying decomposition is independent. This result allowed Hernandez et al. \cite{Hernandez2023} and Hernandez and Buendicho \cite{Hernandez2023powerlaw} to use the method of Johnston et al. \cite{JMP2019:parametrization} on each of the subnetworks (endowed with mass-action kinetics and power-law kinetics, respectively). The parametrized steady states of each subnetwork are then combined to obtain the positive {steady states} of the original network.

Here in our current work, we consider networks that follow kinetics outside the previously considered kinetics (i.e., mass-action and power-law). For instance, we consider the case when reactions have different kinetics, e.g., some reactions follow the mass-action rate law, while others follow a certain quotient rate law.

\section{Results and discussion}
\label{sec:results}

We propose the following general step-by-step procedure to determine the parametrized {steady states} of a system:

\begin{enumerate}
    \item[{S1}.] Decompose the given CRN into independent subnetworks.
    \item[{S2}.] For each subnetwork, either 
       \begin{enumerate}
        \item[{a}.] use the method of Johnston et al. (Theorem \ref{thm:parametrization} in Appendix \ref{GCRN}) if it follows mass-action kinetics, or
        \item[{b}.] { {compute the {positive} steady states directly from the ODEs.}} 
       \end{enumerate}
    \item[{S3}.] Combine the positive steady states of the subnetworks to get the positive {steady states} of the whole network.
\end{enumerate}

{
\begin{remark}
Jonhston et al. \cite{Johnston2014} provided a linear programming method to assist in finding weakly reversible and deficiency zero networks. Furthermore, Hong et al. \cite{HongSIAM} provided a characterization on when a CRN can be translatable into a weakly reversible network by investigating the kernel of the stoichiometric matrix of the CRN and bounds on checking if the translated network has deficiency zero. A MATLAB computational package was built to facilitate the process. Importantly, network translation was applied to several networks \cite{Johnston2014,JMP2019:parametrization,Hernandez2023,MJEB2019,HongSIAM,Hong2021:CommBio}.
\end{remark}
}

What makes our new approach different from previous works \cite{Hernandez2023powerlaw,Hernandez2023} is the second step. Previous approaches apply the method of Johnston et al. to all subnetworks since each subnetwork follows mass-action or power-law kinetics. Our new approach applies the method of Johnston et al. only on subnetworks with mass-action (or power-law) kinetics. For those that follow {non-mass-action} (or non-power-law) kinetics such as quotient kinetics, we compute the steady states directly from the ODEs. 
We apply this approach to a complex mathematical network that involves reactions with {non-mass-action} rate laws.


\subsection{Computation of positive steady states of a complex insulin signaling pathway network with mixed kinetics via network decomposition}


We now illustrate our proposed method by solving for the positive steady state parametrization of a model of the insulin signaling in type 2 diabetes {\cite{BrannmarkInsulin2013}}. The CRN of the model has 36 reactions and 27 species (Figure \ref{fig:method}a). We let $k_i$ be the rate constant of reaction $R_i$. Furthermore, {we} let $x_j$ be the concentration of species $X_j$, {keeping the indices of the species as presented in \cite{LML2023} (see Appendix \ref{app:vars} for a list of the variables used in the model and their definition).}

{The insulin signaling cascade is activated by insulin reception by insulin receptors whose different states are represented by species $X_2$, $X_3$, $X_4$, $X_6$, and $X_7$. Intermediate reactions lead to translocation of the insulin-regulated glucose transporter GLUT4 in the cytosol ($X_{20}$) to the plasma membrane ($X_{21}$), allowing glucose uptake by the cell. Glucose is the main energy currency of the cell; thus, a functioning insulin signaling cascade is necessary for cells to function efficiently.}

The list of reactions (left) with the associated rate laws (right) are given as follows:
\allowdisplaybreaks
\begin{align*}
    R_1 &: X_2 \to X_3   &k_1 x_2\\
    R_2 &: X_2 \to X_4  &k_2 x_2\\
    R_3 &: X_3 \to X_4  &k_3 x_3\\
    R_4 &: X_4 \to X_7   &k_4 x_4\\
    R_5 &: X_7 +X_{25} \to X_6 +X_{25}  &k_5 x_7 x_{25}\\
    R_6 &: X_4 \to X_2  &k_6 x_4\\
    R_7 &: X_6 \to X_2   &k_7 x_6\\
    R_8 &: X_7+X_9 \to X_7+X_{10}  &k_8 x_7 x_9\\
    R_9 &: X_{10} \to X_9  &k_9 x_{10}\\
    R_{10} &: X_{10}+X_{31} \to X_{22}+X_{31}   &k_{10} x_{10} x_{31}\\
    R_{11} &: X_{22} \to X_{10}  &k_{11} x_{22}\\
    R_{12} &: X_{22} \to X_{33}  &k_{12} x_{22}\\
    R_{13} &: X_9 \to X_{23}  &k_{13} x_9\\
    R_{14} &: X_{23} \to X_9   &k_{14} x_{23}\\
    R_{15} &: X_{10} +X_{24} \to X_{10} +X_{25}  &k_{15} x_{10} x_{24}\\
    R_{16} &: X_{25} \to X_{24}  &k_{16} x_{25}\\
    R_{17} &: X_{10}+X_{26} \to X_{10}+X_{27}   &k_{17} x_{10}x_{26}\\
    R_{18} &: X_{27} \to X_{26}  &k_{18} x_{27}\\
    R_{19} &: X_{27}+X_{33} \to X_{29}+X_{33}  &k_{19} x_{27}x_{33}\\
    R_{20} &: X_{22}+X_{28} \to X_{22}+X_{29}   &k_{20} x_{22} x_{28}\\
    R_{21} &: X_{29} \to X_{28}  &k_{21} x_{29}\\
    R_{22} &: X_{28} \to X_{26}  &k_{22} x_{28}\\
    R_{23} &: X_{29}+X_{30} \to X_{29}+X_{31}  &k_{23} x_{29}x_{30}\\
    R_{24} &: X_{27}+X_{30} \to X_{27}+X_{31}   &k_{24} x_{27}x_{30}\\
    R_{25} &: X_{31} \to X_{30}  &k_{25} x_{31}\\    R_{26} &: X_{7}+X_{32} \to X_{7}+X_{33}  &k_{26} x_7 x_{32}\\
    R_{27} &: X_{33} \to X_{32}   &k_{27} x_{33}\\
    R_{28} &: X_{29}+X_{34} \to X_{29}+X_{35}  &k_{28} k_{29} x_{34}\\
    R_{29} &: X_{28}+X_{34}  \to X_{28}+X_{35}   &k_{29} \dfrac{x_{28}^\alpha}{\bar{k}^\alpha+x_{28}^\alpha}x_{34}\\
    R_{30} &: X_{35} \to X_{34}   &k_{30} x_{35}\\
    R_{31} &: X_{35}+X_{20} \to X_{35}+X_{21}  &k_{31} x_{35} x_{20}\\
    R_{32} &: X_{21} \to X_{20}  &k_{32} x_{21}\\
    R_{33} &: X_{31}+X_{36} \to X_{31}+X_{37}  &k_{33} \dfrac{x_{31}^\beta}{\tilde{k}^\beta+x_{31}^\alpha}x_{36}\\
    R_{34} &: X_{37} \to X_{36}   &k_{34} x_{37}\\
    R_{35} &: X_{39} \to X_{38}  &k_{35} x_{39}\\
    R_{36} &: X_{37}+X_{38} \to X_{37}+X_{39}  &k_{36} x_{37} x_{38}
\end{align*}
{where $\tilde{k}$, $\alpha$, and $\beta$ are constants.}

Br\"{a}nnmark et al modeled the activation of AS160 by the serine-473 phosphorylated PKB and the activation of S6K by mTORC1 using Hill-type kinetics \cite{BrannmarkInsulin2013}. These activations, corresponding to reactions $R_{29}$ and $R_{33}$, respectively, are the only reactions that follow non-mass-action kinetics. The rest follow mass-action kinetics. Hence, the network follows mixed kinetics, which is a combination of mass-action and quotient rate laws.


\begin{figure}[!h]
{
\footnotesize
\begin{center}
\includegraphics[width=18cm,height=16cm,keepaspectratio]{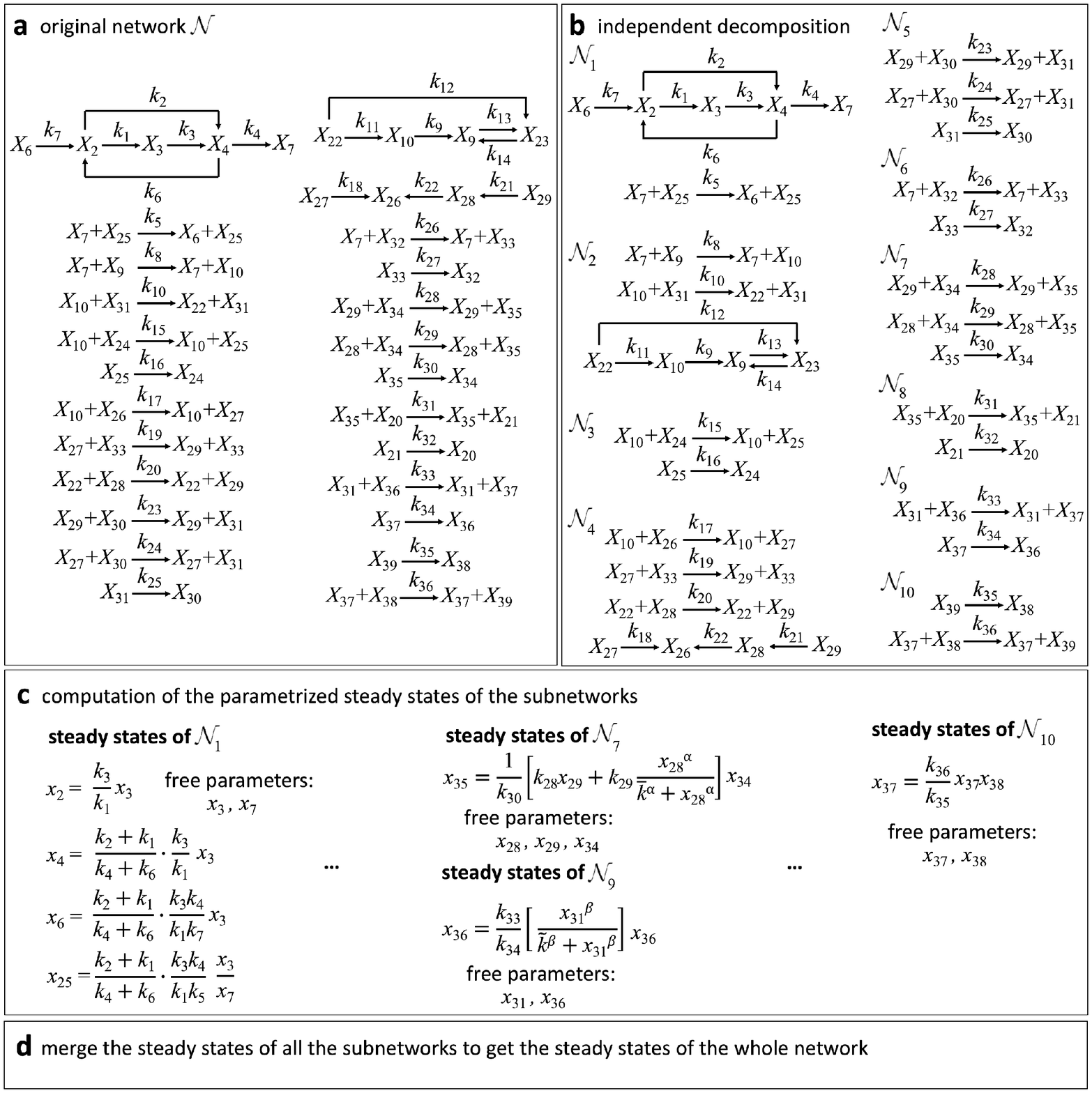}
\end{center}
}
\caption{{{\bf{Parametrization of positive steady states of CRNs with mixed kinetics.}}
{\bf{a}} A CRN $\mathcal{N}$ of a complex mathematical model of insulin signaling in type 2 diabetes.
{\bf{b}} $\mathcal{N}$ is decomposed into 10 independent subnetworks $\mathcal{N}_1, \ldots, \mathcal{N}_{10}$. Here, all the subnetworks except $\mathcal{N}_7$ and $\mathcal{N}_9$ follow mass-action kinetics. On the other hand, subnetworks $\mathcal{N}_7$ and $\mathcal{N}_9$ involve reactions that follow a quotient rate function. 
{\bf{c}} The steady state parametrization is computed for each subnetwork. For subnetworks $\mathcal{N}_1, \ldots, \mathcal{N}_{6}, \mathcal{N}_{8}, \mathcal{N}_{10}$, the parametrization process for mass-action is employed. In particular, we illustrate this process via subnetwork $\mathcal{N}_4$ in Figure \ref{fig:parametrization}. For subnetworks $\mathcal{N}_7$ and $\mathcal{N}_9$ with just three and two reactions, respectively, the steady states are computed using elementary ways of solving equations.
{\bf{d}} The steady states of {all the} subnetworks are merged to get the {steady states} of the whole network.
}
}
\label{fig:method}
\end{figure}

We solve for the positive {steady states} of the network by following the step-by-step procedure we introduced at the beginning of Section \ref{sec:results}. Applying Step S1, we decompose the network into independent subnetworks (Figure \ref{fig:method}b).
To easily get the finest independent decomposition (i.e., independent decomposition with the maximum number of subnetworks), we use the MATLAB program in \cite{LubeniaINDECS}. We enter all the reactions in the network and run the program to generate the output giving the independent decomposition.

We then compute the positive steady states of the subnetworks (Step S2). Among the 10 subnetworks, eight of them ($\mathcal{N}_1, \ldots, \mathcal{N}_{6}, \mathcal{N}_{8}$, and $\mathcal{N}_{10}$) follow entirely the mass-action law. Hence, the parametrization method for mass-action is employed (Step S2a). We illustrate this method for subnetwork $\mathcal{N}_4$ (Figure \ref{fig:parametrization}).

To apply Theorem \ref{thm:parametrization}, we determine a translated network that is both weakly reversible and $V^\star$-directed (Figure \ref{fig:parametrization}a upper right). The translated network has effective and kinetic deficiency of zero, and hence the positive {steady states} can be easily parametrized as described in Figure \ref{fig:parametrization}.

On the other hand, the remaining two subnetworks ($\mathcal{N}_7$ and $\mathcal{N}_9$) follow mixed kinetics. In particular, reactions $R_{29}$ and $R_{33}$ that follow quotient kinetics belong to subnetworks $\mathcal{N}_7$ and $\mathcal{N}_9$, respectively. In this case, the positive steady states of the said subnetworks are computed manually (Step S2b). Finally, the positive steady {state} of the whole network is derived by merging the positive steady states of the subnetworks (Step S3). In particular, if two subnetworks have the same species, then there are two expressions for the steady state formula of such species. We equate these expressions to reduce parameters and for the steady state formula to agree with both subnetworks.

Upon combining all the solutions and expressing the $\sigma$'s and $\tau$'s in terms of the species steady state variables, we get the following parametrized positive {steady states} where the species depend on the free parameters $x_3$, $x_7$, $x_{10}$, $x_{21}$, $x_{26}$, $x_{31}$, $x_{33}$, $x_{34}$, $x_{36}$, and $x_{38}$; and constants $\tilde{k}$, $\alpha$, and $\beta$:

{

\begin{align*}
    x_2 =&  \frac{k_3}{k_1} x_3 \\
    x_4 =&  \frac{k_3 (k_2 + k_1)}{k_1 (k_4 + k_6)} x_3 \\
    x_6 =&  \frac{k_4 k_3 (k_2 + k_1)}{k_7 k_1 (k_4 + k_6)} x_3 \\
    x_9 =&  \left( \frac{k_9}{k_8} + \frac{k_{12} k_{10}}{k_8 (k_{11} + k_{12})} x_{31} \right) \frac{x_{10}}{x_7} \\
    x_{20} =& \frac{k_{32}}{k_{31}} \frac{x_{21}}{x_{35}} \\
    x_{22} =&  \frac{k_{10}}{k_{11} + k_{12}} x_{10} x_{31} \\
    x_{23} =&  \frac{k_{13} k_9}{k_{14} k_8} \frac{x_{10}}{x_7} + \left( \frac{k_{13}}{k_8} \frac{1}{x_7} + 1 \right) \frac{k_{12} k_{10}}{k_{14} (k_{11} + k_{12})} x_{10} x_{31} \\
    x_{24} =&  \frac{k_{16} k_4 k_3 (k_2 + k_1)}{k_{15} k_5 k_1 (k_4 + k_6)} \frac{x_3}{x_{10} x_7} \\
    x_{25} =&  \frac{k_4 k_3 (k_2 + k_1)}{k_5 k_1 (k_4 + k_6)} \frac{x_3}{x_7} \\
    x_{27} =&  \frac{k_{17}}{k_{18} + k_{19} x_{33}} x_{26} x_{10} \\
    x_{28} =&  \frac{k_{19} k_{17}}{k_{22} (k_{18} + k_{19} x_{33})} x_{26} x_{10} x_{33} \\
    x_{29} =&  \left( 1 + \frac{k_{20} k_{10}}{k_{22} (k_{11} + k_{12})} x_{10} x_{31} \right) \frac{k_{19} k_{17}}{k_{21} (k_{18} + k_{19} x_{33})} x_{26} x_{10} x_{33} \\
    x_{30} =&  \frac{k_{25}}{\left( 1 + \frac{k_{20} k_{10}}{k_{22} (k_{11} + k_{12})} x_{10} x_{31} \right) \left( \frac{k_{17}}{k_{18} + k_{19} x_{33}} x_{26} x_{10} \right) \left( \frac{k_{23} k_{19}}{k_{21}} x_{33} + k_{24} \right)} x_{31} \\
    x_{32} =&  \frac{k_{27}}{k_{26}} \frac{x_{33}}{x_7} \\
    x_{35} =&  \left( 1 + \frac{k_{20} k_{10}}{k_{22} (k_{11} + k_{12})} x_{10} x_{31} \right) \frac{k_{28} k_{19} k_{17}}{k_{30} k_{21} (k_{18} + k_{19} x_{33})} x_{26} x_{10} x_{33} x_{34} \\ & + \frac{k_{29}}{k_{30}} \frac{\left( \frac{k_{19} k_{17}}{k_{22} (k_{18} + k_{19} x_{33})} x_{26} x_{10} x_{33} \right)^{\alpha}}{\bar{k}^{\alpha} + \left( \frac{k_{19} k_{17}}{k_{22} (k_{18} + k_{19} x_{33})} x_{26} x_{10} x_{33} \right)^{\alpha}} x_{34} \\
    x_{37} =&  \frac{k_{33}}{k_{34}} \frac{x_{31}^{\beta}}{\tilde{k}^{\beta} + x_{31}^{\beta}} x_{36} \\
    x_{39} =&  \frac{k_{36} k_{33}}{k_{35} k_{34}} \frac{x_{31}^{\beta}}{\tilde{k}^{\beta} + x_{31}^{\beta}} x_{38} x_{36}
\end{align*}
}

By inspection of the steady state formulas $x_i$ of species $X_i$, the system has no absolute concentration robustness (ACR) in {any} species because each steady state value depends on some free parameters. For instance, $x_2$ depends on $x_3$. Similarly, $x_{37}$ depends on $x_{31}$ and $x_{36}$.
ACR is a concept introduced by Guy Shinar and Martin Feinberg \cite{ShinarFeinberg2010} in the journal Science in 2010.
{A system with ACR on a particular species means that the value of the positive steady state for that species does not depend on any initial condition and hence the same for every positive steady state.}




{As mentioned by Lubenia and colleagues} \cite{LML2023}, {the parametrized steady states can potentially be utilized to collaborate with biologists, e.g., by identifying whether species concentrations can be modified to ensure they remain at least approximately constant, maintaining their ACR property and the efficient process they are involved in.}


{Before we proceed with an important result about merging of steady states of the subnetworks of a given network in Step S3, we define the following term.}

{
\begin{definition}
    Let $\mathcal{N}$ be a CRN, and $\mathcal{N}_1$ and $\mathcal{N}_2$ be subnetworks of $\mathcal{N}$. Subnetworks $\mathcal{N}_1$ and $\mathcal{N}_2$ are {\emph{mutually exclusive}} if the species in $\mathcal{N}_1$ do not appear in any complex or any reaction in $\mathcal{N}_2$, and the species in $\mathcal{N}_2$ do not appear in any complex or any reaction in $\mathcal{N}_1$.
\end{definition}
}

\begin{figure}[H]
{
\footnotesize
\begin{center}
\includegraphics[width=15cm,height=14cm,keepaspectratio]{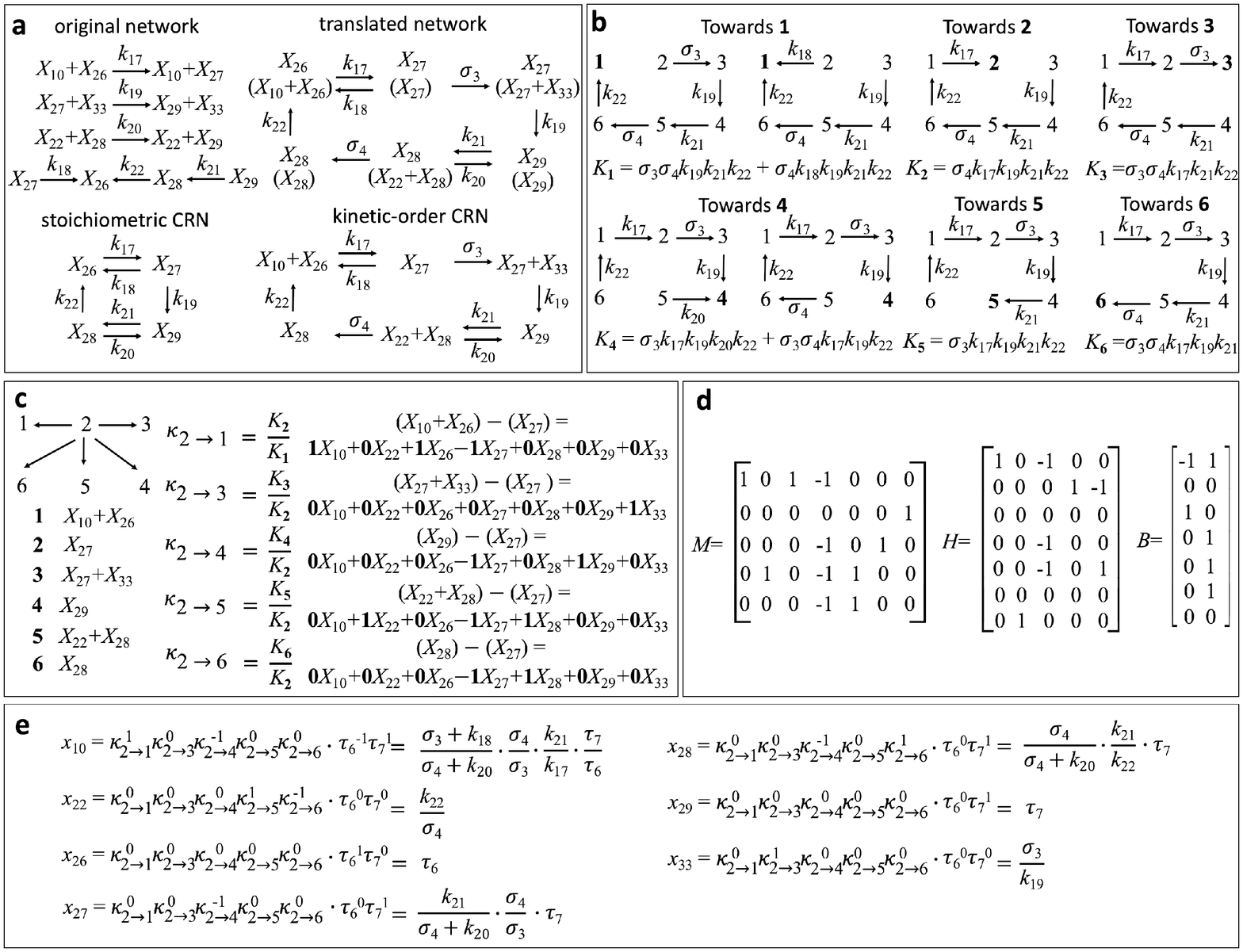}
\end{center}
}
\caption{{\bf Parametrization of positive steady states via network translation.}
{
{\bf{a}} The original CRN (upper left) is translated to a dynamically equivalent weakly reversible generalized network (upper right). The translated network has two substructures: the stoichiometric CRN (lower left) and the kinetic-order CRN (lower right). If the deficiencies of the stoichiometric and kinetic-order CRNs are both zero, proceed with parametrization of steady states.
{\bf{b}} Get all the spanning trees towards each node of the kinetic-order CRN where the kinetic complexes $X_{10}+X_{26}$, $X_{27}$, $X_{27}+X_{33}$, $X_{29}$, $X_{22}+X_{28}$, and $X_{28}$ are labeled 1, 2, 3, 4, 5, and 6, respectively. Compute the tree constant $K_i$ by multiplying the rate constants associated with the edges of a spanning tree. Then get the sum of the products over all the spanning trees towards complex node $i$. 
{\bf{c}} Choose an arbitrary tree containing all the nodes of the kinetic-order CRN (left). Compute $\kappa_{i\to i'}=\dfrac{K_{i'}}{K_{i}}$, i.e., the quotient between the tree constants directed towards $i'$ and $i$ (middle). Additionally, get the difference between the kinetic nodes associated with each edge (right).
{\bf{d}} Construct matrix $M$ such that the rows are the coefficients in the kinetic differences, $H$ is a generalized inverse of $M$ ($MHM=M$) and the ${\rm ker \ } M = B$.
{\bf{e}} Compute the parametrized {steady states} of the original network.
}
}
\label{fig:parametrization}
\end{figure}


{
\begin{proposition}
\label{prop:combine}
Let $\mathcal{N}$ be a CRN decomposed into its finest independent subnetworks $\mathcal{N}_1, \ldots, \mathcal{N}_m,\mathcal{N}_{m+1},\ldots,\mathcal{N}_k$ such that $\mathcal{N}_1, \ldots, \mathcal{N}_m$ have purely mass-action kinetics while $\mathcal{N}_{m+1},\ldots,\mathcal{N}_k$ have non-mass-action kinetics (e.g., mixed kinetics). Then
\begin{enumerate}
    \item $\mathcal{N}^M=\displaystyle \bigcup_{i=1}^m \mathcal{N}_i$ is purely mass-action and the largest mass-action independent subnetwork of $\mathcal{N}$ while $\mathcal{N}^C= \displaystyle \bigcup_{i=m+1}^k \mathcal{N}_i$ has non-mass-action kinetics.
    \item Suppose further that $\mathcal{N}^M$ and $\mathcal{N}^C$ are mutually exclusive and both of them have positive steady states for particular rate constants. Then the whole network $\mathcal{N}=\mathcal{N}^M \cup \mathcal{N}^C$ has positive steady states for the same rate constants considered for $\mathcal{N}^M$ and $\mathcal{N}^C$.
\end{enumerate}
\end{proposition}

\begin{proof}
    (1) directly follows from the assumption that $\mathcal{N}_1,\ldots,\mathcal{N}_m$ are purely mass-action while $\mathcal{N}_{m+1},\ldots,\mathcal{N}_k$ are not. (2) Let $n$ and $p-n$ be the numbers of species of $\mathcal{N}^M$ and $\mathcal{N}^C$, respectively. Since $\mathcal{N}^M$ and $\mathcal{N}^C$ are mutually exclusive, then their steady states can be written in the form $(a_1,\ldots,a_n,a_{n+1},\ldots,a_p)$ and $(b_1,\ldots,b_n,b_{n+1},\ldots,b_p)$, respectively, where $a_{n+1}$, $\ldots,a_p,b_1,\ldots,b_n$ are free parameters, i.e., any positive real numbers. Hence, the positive steady states of the whole network can be written in the form $(a_1,\ldots,a_n,b_{n+1},\ldots,b_p)$.
\end{proof}

Proposition \ref{prop:combine} guarantees that whenever the largest mass-action independent subnetwork ($\mathcal{N}_M$), obtained by getting the union of all the mass-action subnetworks under independent decomposition, and its complement ($\mathcal{N}_C$) are mutually exclusive, then the positive {steady states} of the whole network {exist}, provided that the positive steady steady states of $\mathcal{N}_M$ and $\mathcal{N}_C$ exist. On the other hand, it is difficult to determine if the whole network has positive {steady states} when $\mathcal{N}_M$ and $\mathcal{N}_C$ are not mutually exclusive. This is an open problem that can be studied in the future.
}


\subsection{A remark on the computation of positive steady states via a transformation of a system to its associated mass-action system}

{In some instances, one can find an equivalent mass-action system for a network with mixed kinetics. The following is an interesting example from Yu and Craciun \cite{Yu2018}. They considered a CRN $\mathcal{N}$ with two reactions having the following associated rate functions:
\begin{align}
&R_1:X_1+X_2 \to 2X_1 & & \dfrac{k_1 x_1 x_2}{k_2 + x_1} \\
&R_2:2X_2 \to X_1+X_2 & &k_3 x_1^2.
\end{align}
The first reaction follows a rational reaction rate function (i.e., Michaelis-Menten rate function) while the second one {follows} the standard mass-action rate function.
Because $R_1$ and $R_2$ have different reaction rate functions, we say that the CRN follows mixed kinetics.
The associated system of ODEs for the CRN under the specified mixed kinetics is
\begin{align}
\label{mixed:kinetics:eq1}
\dfrac{dx_1}{dt}&=\dfrac{k_1 x_1 x_2}{k_2 + x_1}-k_3 x_1^2\\
\label{mixed:kinetics:eq2}
\dfrac{dx_2}{dt}&=-\dfrac{k_1 x_1 x_2}{k_2 + x_1}+k_3 x_1^2.
\end{align}

According to Yu and Craciun \cite{Yu2018}, one can instead study the following system of ODEs (although the authors did not specify in detail how they arrived at this equivalent system)
\begin{align}
\dfrac{dx_1}{dt}&=k_1 x_1 x_2 -k_2 k_3 x_1^2 - k_3 x_1^3
\label{eqnotrational1}
\\
\dfrac{dx_2}{dt}&=-k_1 x_1 x_2 +k_2 k_3 x_1^2 + k_3 x_1^3
\label{eqnotrational2}
\end{align}
instead of the system of ODEs comprised of Equations \ref{mixed:kinetics:eq1} and \ref{mixed:kinetics:eq2}. {Note that Equation \ref{eqnotrational1} can be obtained by multiplying Equation \ref{mixed:kinetics:eq1} by its denominator $k_2+x_1$. Similarly, Equation \ref{eqnotrational2} can be obtained by multiplying Equation \ref{mixed:kinetics:eq2} by the same denominator.}
The CRN $\mathcal{N}'$ associated with this system of ODEs is as follows:
\begin{align}
\label{mass:action:eq:1}
&R_1:X_1+X_2 \to 2X_1 & & k_1 x_1 x_2 \\
\label{mass:action:eq:2}
&R_2:2X_1 \to X_1 + X_2 & &k_2' x_1^2 \\
\label{mass:action:eq:3}
&R_3:3X_1 \to 2X_1 + X_2 & &k_3 x_1^3.
\end{align}
Note that when we multiply the ODEs of $\mathcal{N}'$ by $k_2+x_1$, we get precisely the ODEs associated to the mass-action kinetics where $k_2'=k_2k_3$.


COMPILES (COMPutIng anaLytic stEady States) is a computational package \cite{Hernandez2023} developed in MATLAB to compute the analytic positive steady states of a mass-action system and is available at
\url{https://github.com/Mathbiomed/COMPILES}. The user simply inputs all the reactions of a given network and the program then returns the positive steady state parametrization of the network. Entering Equations \ref{mass:action:eq:1}, \ref{mass:action:eq:2}, and \ref{mass:action:eq:3} into the program, we obtain the following parametrization:
\[x_1=\dfrac{\sigma}{k_3} \text{ and } x_2=\dfrac{\sigma (k_2'+\sigma)}{k_1 k_3}=\dfrac{\sigma (k_2 k_3+\sigma)}{k_1 k_3}\]
where $\sigma>0$.
One can easily check that the parametrization satisfies the original ODEs with Equations \ref{mixed:kinetics:eq1} and \ref{mixed:kinetics:eq2}. 





This short discussion leads to the open problem of determining associated ODEs in the format of mass-action. In particular, one may explore the conditions when a given system has an associated (at least with the same {steady states}) or dynamically equivalent mass-action system.
We also note that for complex networks, decomposition of the network into independent subnetworks can also be employed.
    }

\section{Summary and recommendation}
\label{sec:summary}

We proposed a framework to derive positive steady states of CRNs (with nontrivial independent decomposition) that follow {non-mass-action} kinetics, which may include polynomials, quotients, and mixed kinetics. This is done by modifying the algorithm in previous works that focus on mass-action and power-law kinetics \cite{Hernandez2023powerlaw,Hernandez2023}. In particular, after decomposing a given CRN into independent subnetworks, we compute the positive steady states of subnetworks that follow purely mass-action (or power-law) kinetics. On the other hand, we manually compute the positive steady states of subnetworks that follow other types of kinetics. We illustrated this approach via a complex mathematical model of insulin signaling in type 2 diabetes {\cite{BrannmarkInsulin2013}}. Using the obtained steady state parametrization, we were able to check the ACR property for each species in the given network. It would be very exciting to study chemical and biological properties of important systems existing in literature using our approach.
Furthermore, we have initiated looking into transformations of non-mass-action systems to mass-action. It is interesting to explore research in this direction.





\appendix

\section{Feinberg Decomposition Theorem}
\label{Feinberg:Decomposition:Theorem}
In this paper, we focus on independent decomposition due to the following result by Martin Feinberg \cite{Feinberg1987stability,FeinbergBook2019}.

\begin{theorem}
	\label{feinberg:decom:thm}
	Let $\mathcal{N}$ be a CRN (with kinetics $\mathcal{K}$) decomposed into subnetworks $\mathcal{N}_1, \mathcal{N}_2, \ldots, \mathcal{N}_\alpha$. Furthermore, let $\mathcal{K}_1, \mathcal{K}_2, \ldots, \mathcal{K}_\alpha$ be the restriction of $\mathcal{K}$ to reactions in $\mathcal{N}_1, \mathcal{N}_2, \ldots, \mathcal{N}_\alpha$, respectively.
 Then 
	\[E_+^1 \cap E_+^2 \cap \ldots \cap E_+^\alpha \subseteq E_+\]
 where $E_+^i$ is the set of positive steady states of subnetwork $\mathcal{N}_i$.
	If the network decomposition is independent, then equality holds, i.e.,
	\[E_+^1 \cap E_+^2 \cap \ldots \cap E_+^\alpha = E_+.\]
\end{theorem}

\section{A detailed discussion on translated networks}
\label{GCRN}
In a GCRN, \emph{phantom edges} are those
that connect the same stoichiometric complexes. On the other hand, the \emph{effective edges} are those that connect different stoichiometric complexes. We denote by $E^0$ and $E^*$ the sets of phantom edges and effective edges in the GCRN, respectively.
The reaction vector associated with a phantom edge is zero since this edge connects similar stoichiometric complexes. Thus, phantom edges do not contribute to the ODEs of the system. The dummy rate constants associated with these edges are considered \emph{free parameters} and form a vector which we denote by $\sigma$. Furthermore, we denote by $k^*$ the vector of rate constants associated with the effective edges.

We need one more concept to state the theorem of parametrization of positive steady states \cite{JMP2019:parametrization}. This concept is the \textit{$V^\star$-directed networks}.
Here, we form equivalence classes such that each class contains the vertices with identical stoichiometric complex. A representative is then chosen for each class. 
In a $V^\star$-directed network, each effective edge $i\to j$ enters at a representative vertex. In other words, the product node $j$ is associated with a representative vertex. In addition, a phantom edge starts from a representative vertex and then {enters} another vertex within the class.
The reader can check \cite{JMP2019:parametrization} for more details. 

{
\begin{example}
Consider the following CRN endowed with mass-action kinetics 
\cite{Gunawardena2003}:
\begin{center}
\begin{tikzpicture}
        \tikzset{vertex/.style = {minimum size=2em}}
        \tikzset{edge/.style = {->,> = {Stealth[length=2mm, width=2mm]}}}
        \node[vertex] (A) at (0,0) {$A+E$};
        \node[vertex] (AE) at (2.2,0) {$AE$};
        \node[vertex] (BE) at (4.4,0) {$B+E$};
        \node[vertex] (B) at (6,0) {$B$};
        \node[vertex] (Z) at (8.2,0) {$0$};
        \node[vertex] (A1) at (10.4,0) {$A$};
        \draw[edge, thick] (A.9) to (AE.165);
        \draw[edge, thick] (AE.194) to (A.350);
        \draw[edge, thick] (AE.15) to (BE.170);
        \draw[edge, thick] (BE.190) to (AE.343);
        \draw[edge, thick] (B) to (Z);
        \draw[edge, thick] (Z) to (A1);
\end{tikzpicture}
\end{center}
Let us denote this network by $\mathcal{N}$.
It consists of two linkage classes. The first linkage class is the underlying network of a reversible Michaelis-Menten system for enzyme kinetics, where the species $A$, $E$, and $B$ denote a substrate, an enzyme, and a product, respectively. The second linkage class takes into account the depletion of product $B$ and replenishment of substrate $A$.

By adding $E$ to each complex in the second linkage class, we get $B+E \to E \to A+E$ where the reaction vectors are unchanged. Hence, the two resulting complexes ($B+E$ and $A+E$) coincide with the complexes in the first linkage class, and we obtain the following network, which we denote by $\mathcal{N}'_S$:
\begin{center}
\begin{tikzpicture}
        \tikzset{vertex/.style = {minimum size=2em}}
        \tikzset{edge/.style = {->,> = {Stealth[length=2mm, width=2mm]}}}
        \node[vertex] (A) at (0,0) {$A+E$};
        \node[vertex] (AE) at (2.2,0) {$AE$};
        \node[vertex] (BE) at (4.4,0) {$B+E$};
        \node[vertex] (E) at (2.2,-1.8) {$E$};
        \draw[edge, thick] (A.9) to (AE.165);
        \draw[edge, thick] (AE.194) to (A.350);
        \draw[edge, thick] (AE.15) to (BE.170);
        \draw[edge, thick] (BE.190) to (AE.343);
        \draw[edge, thick] (BE) to (E);
        \draw[edge, thick] (E) to (A);
\end{tikzpicture}
\end{center}
Since the new reaction $B+E \to E$ was obtained from the original reaction $B \to 0$, we transfer the original source complex ($B$) as the kinetic complex that we associate with the stoichiometric complex $B+E$. The original source complex $B$ dictates the rate function of the reaction so that the original and the new systems are dynamically equivalent. Furthermore, the unchanged reaction $B+E \to AE$ and the new reaction $B+E\to E$ share the same stoichiometric complexes, but have different associated kinetic complexes because the original source nodes for $B+E \to AE$ and $B+E\to E$ are $B+E$ and $B$, respectively.  Hence, we split the stoichiometric complex ($B+E$) with different kinetic complexes ($B+E$ and $B$) by introducing a phantom edge, which has no effect on the system's ODEs. Hence, we obtain the following translated network, which we denote by $\mathcal{N}'$.

\begin{center}
\begin{tikzpicture}
        \tikzset{vertex/.style ={rectangle, draw, minimum width =10pt,{minimum size=2em}}}
        \tikzset{edge/.style = {->,> = {Stealth[length=2mm, width=2mm]}}}
        \node[vertex,text width=1.12cm,text centered] (A) at (0,0) {{ \circled{1} \ \ \ \ \ $A+E$ \ \ $(A+E)$}};
        \node[vertex,text width=1.12cm,text centered] (AE) at (3,0) {{ \circled{2} \ \ \ \ \ $AE$ \ \ $(AE)$}};
        \node[vertex,text width=1.12cm,text centered] (BE) at (6,0) {{ \circled{3} \ \ \ \  $B+E$ \ \ $(B+E)$}};
        \node[vertex,text width=1.12cm,text centered] (E) at (1.5,-3.5) {{ \circled{5} \ \ \ \  $E$ \ \ \ $(0)$}};
        \node[vertex,text width=1.12cm,text centered] (B) at (4.5,-3.5) {{ \circled{4} \ \ \ \  $B+E$ \ \ $(B)$}};
        \draw[edge, thick] (A.11) to (AE.168)
        node[above,xshift=-8mm] {$k_1$};
        \draw[edge, thick] (AE.190) to (A.350)
        node[below,xshift=8mm] {$k_2$};
        \draw[edge, thick] (AE.12) to (BE.168)
        node[above,xshift=-8mm] {$k_3$};
        \draw[edge, thick] (BE.190) to (AE.350)
        node[below,xshift=8mm] {$k_4$};
        \draw[edge, thick] (BE) to (B)
        node[right,xshift=7mm,yshift=17mm] {$\sigma$};
        \draw[edge, thick] (E) to (A)
        node[left,xshift=8mm,yshift=-18mm] {$k_6$};
        \draw[edge, thick] (B) to (E)
        node[left,xshift=19mm,yshift=-3mm] {$k_5$};
\end{tikzpicture}
\end{center} 


In the translated network $\mathcal{N}'$, the stoichiometric complexes are indicated without parentheses while the kinetic complexes are indicated with parentheses.
We also label the nodes of the translated network by $1,2,\ldots,5$.

The two structures that make up the translated network are the stoichiometric network $\mathcal{N}'_S$ and the kinetic-order network $\mathcal{N}'_K$. The latter can be obtained by removing all the stoichiometric complexes and leaving all the kinetic complexes in $\mathcal{N}'$, i.e., $\mathcal{N}'_K$ is as follows:

\begin{center}
\begin{tikzpicture}
        \tikzset{vertex/.style ={rectangle, draw, minimum width =10pt,{minimum size=2em}}}
        \tikzset{edge/.style = {->,> = {Stealth[length=2mm, width=2mm]}}}
        \node[vertex,text width=1.12cm,text centered] (A) at (0,0) {{ \circled{1} \\ \ \ \ \ \ $A+E$}};
        \node[vertex,text width=1.12cm,text centered] (AE) at (3,0) {{ \circled{2} \\ \ \ \ \ \ $AE$}};
        \node[vertex,text width=1.12cm,text centered] (BE) at (6,0) {{ \circled{3} \\ \ \ \ \  $B+E$}};
        \node[vertex,text width=1.12cm,text centered] (E) at (1.5,-3.5) {{ \circled{5} \\ \ \ \ \\ $0$}};
        \node[vertex,text width=1.12cm,text centered] (B) at (4.5,-3.5) {{ \circled{4} \\ \ \ \  \\ $B$}};
        \draw[edge, thick] (A.11) to (AE.168)
        node[above,xshift=-8mm] {$k_1$};
        \draw[edge, thick] (AE.190) to (A.350)
        node[below,xshift=8mm] {$k_2$};
        \draw[edge, thick] (AE.12) to (BE.168)
        node[above,xshift=-8mm] {$k_3$};
        \draw[edge, thick] (BE.190) to (AE.350)
        node[below,xshift=8mm] {$k_4$};
        \draw[edge, thick] (BE) to (B)
        node[right,xshift=7mm,yshift=17mm] {$\sigma$};
        \draw[edge, thick] (E) to (A)
        node[left,xshift=8mm,yshift=-18mm] {$k_6$};
        \draw[edge, thick] (B) to (E)
        node[left,xshift=19mm,yshift=-3mm] {$k_5$};
\end{tikzpicture}
\end{center} 

We choose a set of representatives $V^* = \{1,2,3,5\}$, which is a subset of the vertex set $V = \{1,2,3,4,5\}$, so that the translated network ($\mathcal{N}'$) is a $V^\star$-directed network, i.e., the product nodes of the effective edges ($k_1$, $k_2$, \ldots, $k_6$) are representative nodes while the phantom edge ($\sigma$) starts from a representative node (i.e., node $3$) and then enter another node within the class (i.e., node $4$).
\label{additional:example:translated}
\end{example}
}

{
\begin{remark}
Note that a computational package called TOWARDZ (TranslatiOn toward
WeAkly Reversible and Deficiency Zero networks) developed using MATLAB was created \cite{HongSIAM} to facilitate network translation that gives, in particular, weakly reversible and deficiency zero networks, if successful.
\end{remark}
}

We now introduce the following main result in the paper by Johnston et al. \cite{JMP2019:parametrization}.

\begin{theorem}
Consider a translated network that is {\emph{weakly reversible and $V^\star$-directed}}. Let $\mathcal{F}$ be any spanning forest containing all nodes of its kinetic-order CRN.
Moreover, let $M$ be the matrix containing all kinetic differences as rows, where the entries per row are arranged according to the order of the species. Furthermore, let $H$ be a matrix such that $M H M = M$, i.e., a generalized inverse of $M$. Finally, define $B$ such that ${\rm im \ } B = \ker M$ and $\ker B = \{\mathbf{0}\}$. Then,
if {\emph{both the effective deficiency and the kinetic deficiency are zero}}, it follows that the set of positive steady states of the original network is equal to:
\begin{equation*}
{\left\{ \kappa(k^*,\sigma)^{H} \circ \tau^{B^\top} | \sigma \in \mathbb{R}^{E^0}_{>0}, \, \tau \in \mathbb{R}^{m - \tilde s}_{>0} \right\} \neq \varnothing}
\end{equation*}
where $\kappa(k^*,\sigma)^{H} \circ \tau^{B^\top}$ is the Hadamard product with the component of $\kappa$ associated with the edge $i \to i'$ defined as $\kappa_{i\to i'}=\dfrac{K_{i'}}{K_i}$ and the \emph{tree constant} $K_i$ defined as the sum (over all the spanning trees of the kinetic-order CRN towards node $i$) of the products of the rate constants associated with the edges of each spanning tree, {and $\tilde{s}$ is the rank of the kinetic-order CRN}.
\label{thm:parametrization}
\end{theorem}

\begin{remark}
    {Theorem \ref{thm:parametrization} does not work on all mass-action systems and is not the only method available for mass-action systems.}
\end{remark}

{
\begin{example}
Consider the original network and its $V^\star$-directed translated network in Example \ref{additional:example:translated}.
It can be easily shown that the translated network is both weakly reversible and deficiency zero, i.e., the stoichiometric CRN $\mathcal{N}'_S$ and the kinetic-order CRN $\mathcal{N}'_K$ are both weakly reversible and deficiency zero.

We now apply Theorem \ref{thm:parametrization}. We consider any tree (i.e., a connected graph without a cycle) containing all the nodes of the kinetic-order CRN. We take a spanning tree with edges $1\to 2$, $1\to 3$, $1\to 4$, and $1\to 5$.
We then compute the kinetic difference $i'-i$ for each edge $i \to i'$. For example, the associated kinetic difference with the edge $1\to 2$ is $AE-(A+E)=-1A-1E+1AE+0B=[-1,-1,1,0]^\top$. We do this for all the edges.
Hence, the matrix $M$ that contains the coefficients in the kinetic difference of an edge per row is

\begin{equation}
\nonumber
M=\begin{blockarray}{ccccc}
A & E & AE & B & \\
\begin{block}{[cccc]c}
-1 & -1 & 1 & 0 & 1\to 2 \\  
-1 & 0 & 0 & 1 & 1\to 3 \\  
-1 & -1 & 0 & 1 & 1\to 4 \\  
-1 & -1 & 0 & 0 & 1\to 5 \\
\end{block}
\end{blockarray}.
\end{equation}

We then compute a generalized inverse $H$ of $M$, i.e., $M H M = M$, and the matrix $B$ where ${\rm im} B = \ker(M)$. We have
\begin{equation}
\nonumber
H=\begin{blockarray}{cccc}
\begin{block}{[cccc]}
0 & 0 & 1 & 0  \\  
-1 & 1 & 0 & 0  \\  
1 & -1 & 0 & 1  \\  
-1 & 0 & -1 & -1 \\
\end{block}
\end{blockarray} \ {\text{and}} \ 
B=\begin{blockarray}{c}
\begin{block}{[c]}
0  \\  
0  \\  
0  \\  
0  \\
\end{block}
\end{blockarray}.
\end{equation}

Meanwhile, we get all the spanning trees towards each of the nodes in the kinetic-order CRN. The required spanning trees are given in Figure \ref{fig:spanningtreeaddex}.

{
\begin{figure}[H]
{
\footnotesize
\begin{center}
\includegraphics[width=14cm,height=8cm,keepaspectratio]{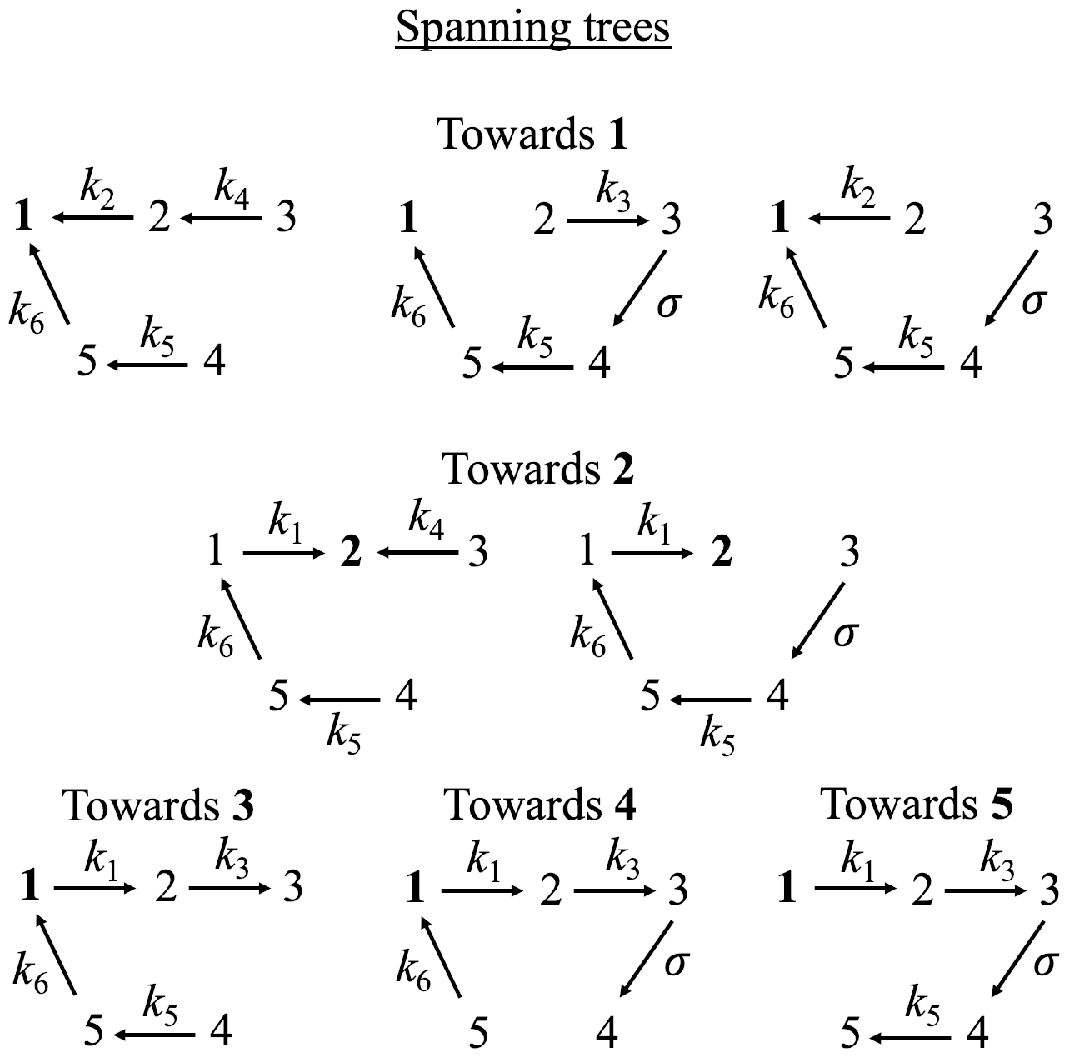}
\end{center}
}
\caption{{{\bf{The spanning trees towards each node in Example \ref{additional:example}.}}
The spanning trees towards the nodes $1,2,\ldots,5$ are enumerated. The edges in the spanning trees are associated with their corresponding rate constants.
}
}
\label{fig:spanningtreeaddex}
\end{figure}
}

Next, we compute the {\it tree constants} $K_i$, for each vertex $i$, of the $V^\star$-directed GCRN given by
 $K_i=\displaystyle \sum _{(\mathcal{V},\mathcal{E}) \in T_i} \prod _{i \to i' \in \mathcal{E}} k_{i \to i'}$ 
 where 
 $T_i$ is the collection of all directed spanning trees rooted at vertex $i$. We then form $\kappa _{i \to i'}=\dfrac{K_{i'}}{K_i}$ for $i \to i' \in \mathcal{R}$.
 The tree constants are
\begin{align*}
    K_1&={k_5 k_6 (k_2 k_4 + k_3 \sigma + k_2 \sigma)}\\
    K_2&={k_1 k_5 k_6 (k_4 + \sigma)}\\
    K_3&={k_1 k_3 k_5 k_6}\\
    K_4&={k_1 k_3 \sigma k_6}\\
    K_5&={k_1 k_3 \sigma k_5}.
\end{align*}
Thus, we get
$\kappa_{1\to 2}=\dfrac{k_1(k_4 + \sigma)}{k_2 k_4 + k_3 \sigma + k_2 \sigma}$, $\kappa_{1\to 3}=\dfrac{k_1 k_3}{k_2 k_4 + k_3 \sigma + k_2 \sigma}$,
$\kappa_{1\to 4}=\dfrac{k_1 k_3 \sigma}{k_5(k_2 k_4 + k_3 \sigma + k_2 \sigma)}$, and $\kappa_{1\to 5}=\dfrac{k_1 k_3 \sigma}{k_6(k_2 k_4 + k_3 \sigma + k_2 \sigma)}$.
Therefore, we obtain the following steady state parametrization of the original network:

\begin{align*}
    a&=\left(\dfrac{k_1 k_3}{k_2 k_4 + k_3 \sigma + k_2 \sigma}\right)^{-1}\cdot\left(\dfrac{k_1 k_3 \sigma}{k_5(k_2 k_4 + k_3 \sigma + k_2 \sigma)}\right)^1\\
    &\cdot\left(\dfrac{k_1 k_3 \sigma}{k_6(k_2 k_4 + k_3 \sigma + k_2 \sigma)}\right)^{-1}\\
    &=\dfrac{k_6}{k_1 k_3 k_5}(k_2 k_4 + k_3 \sigma + k_2 \sigma)\\
    e&=\left(\dfrac{k_1 k_3}{k_2 k_4 + k_3 \sigma + k_2 \sigma}\right)^1\cdot\left(\dfrac{k_1 k_3 \sigma}{k_5(k_2 k_4 + k_3 \sigma + k_2 \sigma)}\right)^{-1}=\dfrac{k_5}{\sigma}\\
    ae&=\left(\dfrac{k_1(k_4 + \sigma)}{k_2 k_4 + k_3 \sigma + k_2 \sigma}\right)^1\cdot\left(\dfrac{k_1 k_3 \sigma}{k_6(k_2 k_4 + k_3 \sigma + k_2 \sigma)}\right)^{-1}=\dfrac{k_6(k_4+\sigma)}{k_3 \sigma}\\
    b&=\left(\dfrac{k_1 k_3 \sigma}{k_5(k_2 k_4 + k_3 \sigma + k_2 \sigma)}\right)^1\cdot\left(\dfrac{k_1 k_3 \sigma}{k_6(k_2 k_4 + k_3 \sigma + k_2 \sigma)}\right)^{-1}=\dfrac{k_6}{k_5}
\end{align*}
where $\sigma >0.$
\label{additional:example}
\end{example}
}

\section{Definition of variables in insulin signaling in type 2 diabetes}\label{app:vars}


The following are the variables used in the model of insulin signaling in type 2 diabetes as presented by \cite{LML2023} based on \cite{BrannmarkInsulin2013}:
\begin{align*}
    & X_2 = \text{inactive receptor} \\
    & X_3 = \text{insulin-bound receptor} \\
    & X_4 = \text{tyrosine-phosphorylated receptor} \\
    & X_6 = \text{internalized dephosphorylated receptor} \\
    & X_7 = \text{tyrosine-phosphorylated and internalized receptor} \\
    & X_9 = \text{inactive IRS1} \\
    & X_{10} = \text{tyrosine-phosphorylated IRS1} \\
    & X_{20} = \text{glucose transporter 4 from the cytosol} \\
    & X_{21} = \text{glucose transporter 4 in the plasma membrane} \\
    & X_{22} = \text{combined tyrosine/serine 307-phosphorylated IRS1} \\
    & X_{23} = \text{serine 307-phosphorylated IRS1} \\
    & X_{24} = \text{inactive negative feedback} \\
    & X_{25} = \text{active negative feedback} \\
    & X_{26} = \text{inactive PKB} \\
    & X_{27} = \text{threonine 308-phosphorylated PKB} \\
    & X_{28} = \text{serine 473-phosphorylated PKB} \\
    & X_{29} = \text{combined threonine 308/serine 473-phosphorylated PKB} \\
    & X_{30} = \text{mTORC1} \\
    & X_{31} = \text{mTORC1 involved in phosphorylation of IRS1 at serine 307} \\
    & X_{32} = \text{mTORC2} \\
    & X_{33} = \text{mTORC2 involved in phosphorylation of PKB at threonine 473} \\
    & X_{34} = \text{AS160} \\
    & X_{35} = \text{AS160 phosphorylated at threonine 642} \\
    & X_{36} = \text{S6K} \\
    & X_{37} = \text{activated S6K; phosphorylated at threonine 389} \\
    & X_{38} = \text{S6} \\
    & X_{39} = \text{activated S6; phosphorylated at serine 235 and serine 236}
\end{align*}

\end{document}